\documentclass{tran-l}
\usepackage{amsmath, amssymb}
\usepackage{tikz-cd}
\usepackage{lineno}
\newtheorem{theorem}{Theorem}[section]

\newtheorem{corollary}[theorem]{Corollary}

\theoremstyle{definition}
\newtheorem{definition}[theorem]{Definition}
\newtheorem{example}[theorem]{Example}

\theoremstyle{remark}
\newtheorem{remark}[theorem]{Remark}

\numberwithin{equation}{section}

\begin{document}

\title{ESSENTIAL SUBGROUPS AND ESSENTIAL EXTENSIONS}

\author{Sourav Koner}
\address{Department of Mathematics, The University of Burdwan, Burdwan Rajbati, West Bengal 713104}
\email{harakrishnaranusourav@gmail.com}

\author{Biswajit Mitra}
\address{Department of Mathematics, The University of Burdwan, Burdwan Rajbati, West Bengal 713104}
\email{bmitra@math.buruniv.ac.in}



\begin{abstract}
The notion of essential submodules and essential extensions of modules are extended to groups (typically nonabelian), and several necessary and sufficient conditions for a group to possess a proper essential subgroup are investigated. Further, we have completely characterized groups that do not possess a proper essential extension. These observations are used in concluding several properties of groups having essential subgroups. Finally, a short proof of the well-known theorem of Eilenberg and Moore that the only injective object in the category of groups is the trivial group is given.
\end{abstract}

\maketitle

\section{Introduction}

A submodule $N$ of a module $M$ is called an essential submodule if it intersects every nontrivial submodule of $M$ non-trivially. The theory of essential submodules can be found in \cite{A, C}. Since abelian groups are $\mathbb{Z}$-modules and subgroups of abelian groups are $\mathbb{Z}$-submodules, sufficient works exist in the literature on essential subgroups of abelian groups \cite{B} but hardly found its study over nonabelian groups. Analogous to the theory of essential submodules, the present paper aims to study essential subgroups of an arbitrary group $G$ (abelian or nonabelian) and determine much of the structure of $G$. The present paper consists of three sections. In section \eqref{aa}, we discuss properties of essential subgroups of a group (not necessarily commutative). Under a certain condition, we exhibit an essential subtgroup under a general action (Theorem \eqref{khma}) of groups. As an application of Theorem \eqref{khma}, we see that the symmetric group $\mathrm{Sym}(X)$ on a set $X$ with $|X| > 2$ has a proper essential subgroup. Also, the well-known fact that $\mathrm{Sym}(X)$ is indecomposable follows easily from this theorem. Further, we see how a local study on a normal subgroup of a finite group assures the parent group to contain a proper essential subgroup (Corollary \eqref{babcho}). Finally, we establish a connection between malnormality and essentiality (Corollary \eqref{jbnkbj}). We further prove (Theorem \eqref{sk}) that the direct sum of a family of groups has proper essential subgroups if and only if at least one of them has so. For a group $G$, if $e(G)$ denotes the intersection of proper essential subgroups of $G$ then it has been shown that $e$ distributes over arbitrary direct sums. A complete answer to the question of when the socle of a group $G$ equals $e(G)$ (Theorem \eqref{jj}) is given, Theorem \eqref{kk} provides a complete description of a group that has no proper essential subgroups. Finally, Theorem \eqref{nkb} of this section narrates the behavior of essentiality under the semi-direct product of groups. Section \eqref{mm} is devoted to the essential extension of a group. Theorem \eqref{ma} provides a necessary and sufficient criterion of a group having no proper essential extension (if and only if the group is complete). As a consequence (Corollary \eqref{omnamosibai}), it follows that every non-trivial abelian group has a proper essential extension. Finally, in section \eqref{mbmba} we give a short proof of Eilenberg-Moore's theorem (Corollary \eqref{rmrm}) which follows from the fact that (Theorem \eqref{ll}) the only extensional object in the category of groups is the trivial group. In the present paper, we use the following theorem that can be found in (theorem 2.1 \cite{K}).

\begin{theorem}\label{bchche}
Let $G$ be a group and $\mathrm{Hol}(G)$ be its holomorph. Then $G$ is a direct summand of $\mathrm{Aut}(G)$ if and only if either $G$ is complete or $G \cong \mathbb{Z}_{2} \times C$, where $C$ is some complete group such that $(C: F) \neq 2$ for all subgroups $F$ of $C$.
\end{theorem}

\section{\large{Essential subgroups}}\label{aa}

The following definition of an essential subgroup can be found in \cite{D}.

\begin{definition}
A normal subgroup $E$ of a group $G$ is called an essential subgroup if every nontrivial normal subgroup of $G$ intersects $E$ nontrivially. An essential subgroup $E$ of $G$ is said to be a proper essential subgroup if $E \neq G$. The intersection of all essential subgroups of $G$ is denoted by $e(G)$. We define $e(G) = \{1\}$ if $G = \{1\}$. 
\end{definition}

Every group $G \neq \{1\}$ has essential subgroups since $G$ is essential in $G$. Moreover, if a normal subgroup $E$ is essential in $G$, then $E = E \cap G \neq \{1\}$.

\begin{example}
The center $\mathrm{Z}(G)$ of a nilpotent group $G$ is a proper essential subgroup of $G$ (theorem $5.2.1$ \cite{G}).
\end{example}

\begin{theorem}\label{khma}
Let $H$ be a subgroup of a nontrivial group $G$ and $f: G \times X \rightarrow X$ be an action. If $H_{x} \nsubseteq H_{y}$ for all $x, y \in X$ such that $x \neq y$, then $\mathrm{ncl}_{G}(H)\mathrm{Ker}(f)$ is an essential subgroup of $G$.
\end{theorem}
\begin{proof}
Let $N$ be a normal subgroup of $G$. If $\mathrm{ncl}_{G}(H)\mathrm{Ker}(f) \cap N = \{1\}$, then $\mathrm{ncl}_{G}(H) \cap N = \{1\}$. This implies that $N \subseteq \mathrm{C}_{G}(H)$. For $\alpha \in X$, we denote $\mathrm{Fix}(H_{\alpha}) = \{x \in X \mid hx = x, \text{ for all } h \in H_{\alpha}\}$. For $g \in \mathrm{C}_{G}(H)$ and $x \in \mathrm{Fix}(H_{\alpha})$, we have $gx = (ghg^{-1})(gx)$, for all $h \in H_{\alpha}$, that is, $h(gx) = gx$, for all $h \in H_{\alpha}$, and so $gx \in \mathrm{Fix}(H_{\alpha})$. But $\mathrm{Fix}(H_{\alpha}) = \{\alpha\}$, since $H_{\alpha} \nsubseteq H_{y}$, for all $y \in X$ such that $\alpha \neq y$. This implies that $g\alpha = \alpha$, for all $g \in \mathrm{C}_{G}(H)$. As $\alpha$ was arbitrary, we see that $\mathrm{C}_{G}(H) \subseteq \mathrm{Ker}(f)$, and this implies that $N = \{1\}$, since $N \subseteq \mathrm{Ker}(f)$.
\end{proof}

\begin{corollary}
Let $\mathrm{Sym}(X)$ be the symmetric group on a set $X$ such that $|X| > 2$. Then $\mathrm{Sym}(X)$ is indecomposable and $\mathrm{Sym}(X)$ contains a proper essential subgroup.
\end{corollary}
\begin{proof}
The statement is obvious when $|X|$ equals $3$ and $4$. Hence, assume that $|X| > 4$. Let $G = \mathrm{Sym}(X)$ and $H$ be the subgroup consisting of $\sigma \in G$ such that $\sigma$ acts as an even permutation on the finite set of points it moves. It is easy to see that $H$ is a normal subgroup of $G$, and so $\mathrm{ncl}_{G}(H) = H$. Let $f: \mathrm{Sym}(X) \times X \rightarrow X$ be the obvious action. It is easy to see that $\mathrm{Ker}(f) = \{1\}$ and $H_{x} \nsubseteq H_{y}$ whenever $x \neq y$. Thus, $\mathrm{ncl}_{G}(H)\mathrm{Ker}(f) = H$ is an essential subgroup of $\mathrm{Sym}(X)$, and it is proper since $H \neq \mathrm{Sym}(X)$. Moreover, it is well-known that $H$ is simple. Hence, the proof of Theorem \eqref{khma} shows that $H \subseteq N$ for all nontrivial normal subgroup $N$ of $\mathrm{Sym}(X)$, and this implies $\mathrm{Sym}(X)$ is indecomposable.
\end{proof}

\begin{corollary}\label{babcho}
If $S$ is a self-normalizing subgroup of a nontrivial finite group $G$, then either $\mathrm{ncl}_{G}(S) = G$ or $\mathrm{ncl}_{G}(S)$ is a proper essential subgroup of $G$. 
\end{corollary}
\begin{proof}
Let $\mathrm{ncl}_{G}(S) = H$ and assume that $H \neq G$. Consider the obvious action of $G$ on $G/S$. If $H_{xS} \subseteq H_{yS}$, then $xSx^{-1} = ySy^{-1}$, and this implies $y^{-1}x \in S$, since $\mathrm{N}_{G}(S) = S$. This shows that $xS = yS$. Since $H\mathrm{Ker}(f) = H$, Theorem \eqref{khma} implies that $H$ is a proper essential subgroup of $G$.  
\end{proof}

\begin{corollary}\label{jbnkbj}
If $S$ is a nontrivial malnormal subgroup of a group $G$, then either  $\mathrm{ncl}_{G}(S) = G$ or $\mathrm{ncl}_{G}(S)$ is a proper essential subgroup of $G$. 
\end{corollary}
\begin{proof}
Let $\mathrm{ncl}_{G}(S) = H$ and assume that $H \neq G$. Consider the obvious action of $G$ on $G/S$. If $H_{xS} \subseteq H_{yS}$, then $xSx^{-1} \subseteq ySy^{-1}$, and this implies $tSt^{-1} \subseteq S$, where $t = y^{-1}x$. Since $S$ is malnormal, we get that $t \in S$, and so $xS = yS$. Hence, Theorem \eqref{khma} implies that $H$ is a proper essential subgroup of $G$.
\end{proof}

\begin{remark}\label{tmrasmkrunankb}
Let $\Lambda$ be a set with $|\Lambda| \geq 2$. For each $\lambda \in \Lambda$, let $K_{\lambda}$ be a nontrivial group and let $G = \ast_{\lambda} K_{\lambda}$ be the free product of the groups $K_{\lambda}$, for all $\lambda \in \Lambda$. It is easy to verify that $K_{\lambda}$ is a malnormal subgroup of $G$, for all $\lambda \in \Lambda$. Hence, Corollary \eqref{jbnkbj} implies that $\mathrm{ncl}_{G}(K_{\lambda})$ is a proper essential subgroup of $G$, for all $\lambda \in \Lambda$, since $\mathrm{ncl}_{G}(K_{\lambda}) \cap K_{\lambda'} = \{1\}$, for all $\lambda' \in \Lambda$ such that $\lambda' \neq \lambda$. We can now conclude that the cardinality of the set of proper essential subgroups of $G$ is at least $|\Lambda|$, since $\mathrm{ncl}_{G}(K_{\lambda}) \neq \mathrm{ncl}_{G}(K_{\lambda'})$ whenever $\lambda \neq \lambda'$, and that $G$ is a product of these proper essential subgroups. 
\end{remark}

\begin{theorem}\label{sk}
Let $G = \bigoplus_{\lambda \in \Lambda} G_{\lambda}$, where $\{G_{\lambda}\}_{\lambda \in \Lambda}$ be a family of groups indexed by a nonempty set $\Lambda$. Then 

$(a)$ $G$ has a proper essential subgroup if and 
only if at least one of the components

\hspace{0.45cm} of $G$ has a proper essential subgroup.

$(b)$ $e(G) = \bigoplus_{\lambda \in \Lambda} e(G_{\lambda})$.
\end{theorem}
\begin{proof}
$(a)$ Let $E$ be a proper essential subgroup of $G$ and $\lambda \in \Lambda$. Observe that $\{1\} \neq E \cap G_{\lambda}$ is a normal subgroup of $G_{\lambda}$. If $N \neq \{1\}$ is a normal subgroup of $G_{\lambda}$, then $N$ is a normal subgroup of $G$ since $G_{\lambda}$ is a direct summand of $G$ and so $N \cap E \neq \{1\}$. Hence $(E \cap G_{\lambda}) \cap N = E \cap N \neq \{1\}$. This shows that $E \cap G_{\lambda}$ is an essential subgroup of $G_{\lambda}$ for all $\lambda \in \Lambda$. If $E \cap G_{\lambda} = G_{\lambda}$ for every $\lambda \in \Lambda$, then $G_{\lambda} \leq E$ for every $\lambda \in \Lambda$ and $E = G$ against the hypothesis. Hence there exists $\lambda_{0} \in \Lambda$ such that $E \cap G_{\lambda_{0}} \neq G_{\lambda_{0}}$ and $G_{\lambda_{0}}$ has an essential proper subgroup.

Conversely suppose that there exists $\lambda_{0} \in \Lambda$ such that $G_{\lambda_{0}}$ has a proper essential subgroup $E_{\lambda_{0}}$. Put $X = \bigoplus_{\lambda \in \Lambda} X_{\lambda}$ with $X_{\lambda} = G_{\lambda}$ for every $\lambda \in \Lambda \setminus \{\lambda_{0}\}$ and $X_{\lambda_{0}} = E_{\lambda_{0}}$. Note that $X$ is a proper subgroup of $G$. We claim that $X$ is an essential subgroup $G$. Let $N \neq \{1\}$ be a normal subgroup of $G$ and consider the projection $\pi_{\lambda_{0}}: G \rightarrow G_{\lambda_{0}}$ on $\lambda_{0}$-th coordinate; if $\pi_{\lambda_{0}}(N) = \{1\}$, then $N \subseteq \bigoplus_{\lambda} Y_{\lambda}$ with $Y_{\lambda} = X_{\lambda}$ if $\lambda \neq \lambda_{0}$ and $Y_{\lambda_{0}} = \{1\}$, then $N \cap X = N \neq \{1\}$ since $N \subseteq X$.

Suppose $\pi_{\lambda_{0}}(N) \neq \{1\}$, then $\pi_{\lambda_{0}}(N)$ is a nontrivial normal subgroup of $G_{\lambda_{0}}$ and $\pi_{\lambda_{0}}(N) \cap E_{\lambda_{0}} = \pi_{\lambda_{0}}(N) \cap X_{\lambda_{0}} \neq \{1\}$. If $1 \neq y \in \pi_{\lambda_{0}}(N) \cap X_{\lambda_{0}}$ and $n = (x_{\lambda})_{\lambda \in \Lambda}$ with $x_{\lambda} = 1$ for every $\lambda \in \Lambda \setminus \{\lambda_{0}\}$ and $x_{\lambda_{0}} = y$, then $1 \neq n \in N \cap X$ and hence $N \cap X \neq \{1\}$.

\vspace{0.1cm}

$(b)$ Put $\overline{G} := \bigoplus_{\lambda \in \Lambda} e(G_{\lambda})$ and notice that if $E$ is an essential subgroup $G$, then $E \cap G_{\lambda}$ is an essential subgroup of $G_{\lambda}$ for all $\lambda \in \Lambda$. Then $e(G_{\lambda}) \subseteq E \cap G_{\lambda} \subseteq E$ for every $\lambda \in \Lambda$, and so $\overline{G} \subseteq E$ for every essential subgroup $E$ of $G$. Hence $\overline{G} \subseteq e(G)$.

To prove the opposite inclusion, notice that if $G_{\lambda}$ has no proper essential subgroups, that is, $e(G_{\lambda}) = G_{\lambda}$, for every $\lambda \in \Lambda$, then from $(a)$ there follows that $G$ has no proper essential subgroups, hence $e(G) = G = \bigoplus_{\lambda \in \Lambda} e(G_{\lambda}) = \overline{G}$. Therefore we may suppose that $T := \{\lambda \in \Lambda \mid e(G_{\lambda}) \neq G_{\lambda}\} \neq \emptyset$. Let $(g_{\lambda})_{\lambda \in \Lambda} \in e(G)$ and $\tau \in T$; if $E_{\tau}$ is an essential subgroup of $G_{\tau}$ then we have that $X = \bigoplus_{\lambda \in \Lambda} X_{\lambda}$, with $X_{\lambda} = G_{\lambda}$ if $\lambda \neq \tau$ and $X_{\tau} = E_{\tau}$, is an essential subgroup of $G$: let $N$ be any nontrivial normal subgroup of $G$; then $\pi_{\lambda}(N) \unlhd G_{\lambda}$ for every $\lambda \in \Lambda$ where $\pi_{\lambda}: G \rightarrow G_{\lambda}$ is the projection on $\lambda$-th coordinate, and, since $N \neq \{1\}$, there exists $\mu \in \Lambda$ such that $\pi_{\mu}(N) \neq \{1\}$; if $\mu \neq \tau$ then it is clear that $X \cap N \neq \{1\}$; suppose $\mu = \tau$; since $E_{\mu}$ is an essential subgroup of $G_{\mu}$, $\pi_{\mu}(N) \cap E_{\mu} \neq \{1\}$ and so $X \cap N \neq \{1\}$. This implies that $(g_{\lambda})_{\lambda \in \Lambda} \in X$ and, in particular $g_{\tau} \in E_{\tau}$ since $e(G) \subseteq X$. So $g_{\tau} \in E_{\tau}$ for every essential subgroup $E_{\tau}$ of $G_{\tau}$ and this implies that $g_{\tau} \in e(G_{\tau})$. Obviously if $\lambda \in \Lambda \setminus T$, then $e(G_{\lambda}) = G_{\lambda}$ and $g_{\lambda} \in e(G_{\lambda})$. Therefore $(g_{\lambda})_{\lambda \in \Lambda} \in \overline{G}$.
\end{proof}

Recall that the socle of a group $G$, denoted by $\mathrm{soc}(G)$, is the subgroup of $G$ generated by all minimal normal subgroups of $G$. If $G$ has no minimal normal subgroups then its socle is defined to be trivial. Note that in a nontrivial finite group $G$ the socle is an essential subgroup of $G$ (see Remark \eqref{pm} for a more general result). Further, note that $\mathrm{soc}(G) \subseteq e(G)$.

\begin{theorem}\label{jj}
Let $G$ be a nontrivial group. Then $\mathrm{soc}(G) = e(G)$ if and only if $\mathrm{soc}(G)$ is an essential subgroup of $e(G)$.
\end{theorem}
\begin{proof}
If $\mathrm{soc}(G) = e(G)$, there is nothing to show. Assume now that $\mathrm{soc}(G)$ is a proper essential subgroup of $e(G)$. Since $\mathrm{soc}(G)$ is a proper subgroup of $e(G)$, $\mathrm{soc}(G)$ is not an essential subgroup of $G$, and so there exists a nontrivial proper normal subgroup $N$ of $G$ such that $N$ is maximal with respect to the property that $\mathrm{soc}(G) \cap N = \{1\}$. Observe that $E = \mathrm{soc}(G)N$ is an essential subgroup of $G$, otherwise, a nontrivial proper normal subgroup $K$ of $G$ would exist such that $E \cap K = \{1\}$ and it would then contradict the maximality of $N$. This implies that $e(G) = \mathrm{soc}(G)(e(G) \cap N)$. Since $e(G) \cap N \unlhd e(G)$ and $\mathrm{soc}(G)$ is an essential subgroup of $e(G)$, we get that $e(G) \cap N = \{1\}$, a contradiction. Hence, $\mathrm{soc}(G) = e(G)$.
\end{proof}

Recall that the length of a group $G$, denoted by $l(G)$, is defined to be the length of a Jordan-H\"{o}lder filtration of $G$. We write $l(G) < \infty$ if $G$ possesses a composition series \cite{H}. It is not hard to see that for a nontrivial group $G$, if $l(G) < \infty$ or if $G$ has a chief series, then every nontrivial normal subgroup of $G$ contains a minimal normal subgroup of $G$.

\begin{remark}\label{pm}
For a nontrivial group $G$, if $l(G) < \infty$ or if $G$ has a chief series, then $\mathrm{soc}(G)$ is an essential subgroup of $G$; moreover, $\mathrm{soc}(G) = e(G)$. Because in the proof of Proposition \eqref{jj}, we have already seen that if $\mathrm{soc}(G) \neq e(G)$, then $e(G) = \mathrm{soc}(G)(e(G) \cap N)$ for some $N \unlhd G$ such that $\mathrm{soc}(G) \cap N = \{1\}$. If $e(G) \cap N$ would be a nontrivial normal group of $G$, then a minimal normal subgroup $K$ of $G$ would exist such that $K \subseteq e(G) \cap N$. But this is absurd since $K \subseteq \mathrm{soc}(G) \cap (e(G) \cap N) = \mathrm{soc}(G) \cap N = \{1\}$. Therefore, it must be that $e(G) \cap N$ is the trivial group.  
\end{remark}

\begin{theorem}\label{kk}
Let $G$ be a nontrivial group. Then the following are equivalent.

$(a)$ $G$ contains no proper essential subgroup.

$(b)$ Every normal subgroup $N$ of $G$ is a direct summand of $G$.

$(c)$ Every nontrivial normal subgroup of $G$ contains no proper essential subgroup.

$(d)$ There exists a group $\overline{G}$ that does not have a proper essential subgroup and

\hspace{0.45cm} contains a normal subgroup isomorphic to $G$.

$(e)$ For every proper normal subgroup $N$ of $G$ and for every normal subgroup $A$ 

\hspace{0.45cm} of $G$ containing $N$, there exists a normal subgroup $B$ of $G$ containing $N$ such

\hspace{0.45cm} that $G = AB$ and $A \cap B = N$. 
\end{theorem}
\begin{proof}
$(a) \Rightarrow (b)$. The claim is true if $G$ is a simple group. So we may assume that $G$ has a nontrivial proper normal subgroup $N$. As $N$ is not an essential subgroup of $G$, there exists a nontrivial proper normal subgroup $\overline{N}$ of $G$ such that $\overline{N}$ is maximal with respect to the property that $N \cap \overline{N} = \{1\}$. Observe that $N \overline{N}$ is an essential subgroup of $G$, otherwise, a nontrivial proper normal subgroup $K$ of $G$ would exist such that $N\overline{N} \cap K = \{1\}$ and it would then contradict the maximality of $\overline{N}$ since $N \cap K\overline{N} = \{1\}$. Hence, $G = N \overline{N}$.

The implications that $(b) \Rightarrow (a)$, $(a) \Rightarrow (d)$, and $(c) \Rightarrow (a)$ are easy to verify. So, we prove $(a) \Rightarrow (c)$ and $(d) \Rightarrow (a)$.

$(a) \Rightarrow (c)$. Assume on the contrary that $N$ is a nontrivial normal subgroup of $G$ and $N$ contains a proper essential subgroup. From the implication $(a) \Rightarrow (b)$, we have that $G = NT$ and $N \cap T = \{1\}$, for some $T \unlhd G$. But Theorem \eqref{sk}$(a)$ implies that $G$ has a proper essential subgroup, a contradiction.

$(d) \Rightarrow (a)$. Assume on the contrary that $G$ contains a proper essential subgroup. Since $\overline{G}$ does not have a proper essential subgroup, therefore, from the implication $(a) \Rightarrow (b)$, we see that there exists a normal subgroup $T$ of $\overline{G}$ such that $G \cap T = \{1\}$ and $\overline{G} = GT$. But Theorem \eqref{sk}$(a)$ implies that $\overline{G}$ has a proper essential subgroup, a contradiction.

$(e) \Rightarrow (a)$. If we take $N = \{1\}$, we see that every normal subgroup $N$ of $G$ is a direct summand of $G$.

$(a) \Rightarrow (e)$ Let $N$ be a proper normal subgroup of $G$. We claim that $G/N$ does not have a proper essential subgroup. Because, on the contrary, if we assume that $E/N$ is a proper essential subgroup of $G/N$, then observe that $E$ is a proper normal subgroup of $G$ and $N \subsetneq E$. We claim that $E$ is an essential subgroup of $G$. Let $H$ be a nontrivial normal subgroup of $G$. If $H \subseteq N$, then clearly $E \cap H \neq \{1\}$, so, assume that $H \nsubseteq N$. As $HN/N \unlhd G/N$ and $HN/N \neq \{1\}$, therefore, there exists $hN \in E/N \cap HN/N$ such that $h \in H$ and $h \notin N$. This implies $h = en$, for some $n \in N$ and $e \in E$ such that $e \notin N$, and so $h \in E \cap H$. But this implies $E$ is a proper essential subgroup of $G$, a contradiction. 

As $G/N$ does not contain a proper essential subgroup, applying $(a) \Rightarrow (b)$ to the group $G/N$, we see that $G/N = (A/N)(B/N)$ and $A/N \cap B/N = \{1\}$, for all $A/N \unlhd G/N$ and for some $B/N \unlhd G/N$. But this implies $AB = G$ and $A \cap B \subseteq N$, for all $A \unlhd G$ such that $N \subseteq A$ and for some $B \unlhd G$ such that $N \subseteq B$. Finally, since $N \subseteq A \cap B$, we get that $N \subseteq A \cap B$, so $A \cap B = N$.
\end{proof}

\begin{remark}
In Theorem \eqref{kk} if we assume that $l(G) < \infty$ or $G$ has a chief series, then we have: $(a) \Leftrightarrow (b) \Leftrightarrow (c) \Leftrightarrow (d) \Leftrightarrow (e) \Leftrightarrow \mathrm{soc}(G) = G$. Because assuming that every normal subgroup $N$ of $G$ is a direct summand of $G$ gives $G = \mathrm{soc}(G)N$, for some $N \unlhd G$ such that $\mathrm{soc}(G) \cap N = \{1\}$. If $N$ would be a nontrivial normal subgroup of $G$, then a minimal normal subgroup $T$ of $G$ would exist such that $T \subseteq N$. But this is absurd since $T \subseteq \mathrm{soc}(G) \cap N = \{1\}$. Therefore, it must be that $N$ is the trivial group. Again, if $\mathrm{soc}(G) = G$, then $G = \bigoplus_{\lambda \in \Lambda} M_{\lambda}$, where each $M_{\lambda}$ is a simple group. Now, if $E$ is an essential subgroup of $G$, then $M_{\lambda} \subseteq E$, for all $\lambda \in \Lambda$. But this implies $G \subseteq E$.
\end{remark}

\begin{theorem}\label{nkb}
The semi-direct product $G = N \rtimes H$ has a proper essential subgroup if at least one of the components of $G$ has a proper essential subgroup. If both $N$ and $H$ are abelian groups such that the action of $H$ on $N$ is nontrivial, then $G$ contains a proper essential subgroup.
\end{theorem}
\begin{proof}
Let $E$ be a proper essential subgroup of $H$, we claim that $L = N \rtimes E$ is an essential subgroup of $G$. Let $K \neq \{1\}$ be a normal subgroup of $G$ and $\pi_{2}: nh \in G \rightarrow h \in H$ be the projection. It is easy to verify that $\pi_{2}(K)$ is normal in $H$. If $\pi_{2}(K) = \{1\}$, then $K \subseteq N$ and $L \cap K = (N \rtimes E) \cap K = K \rtimes E \neq \{1\}$. If $\pi_{2}(K) \neq \{1\}$, then there exists $h \in E \cap \pi_{2}(K) \setminus \{1\}$ because $E$ is an essential subgroup of $H$. Then there exists $b \in N$ such that $bh \in K$, hence $1 \neq bh \in K \cap L$ since $bh \in N \rtimes E$.

If $N$ has a proper essential subgroup, then Theorem \eqref{kk} implies that $G$ has a proper essential subgroup.

Suppose that $H$ and $N$ are abelian and that the action of $H$ on $N$ is nontrivial, that is $G$ is nonabelian. If $G$ does not have a proper essential subgroup then, from Theorem \eqref{kk}, there follows that $N$ is a direct summand of $G$. Hence there exists a normal subgroup $T$ of $G$ such that $G = NT$ and $N \cap T = \{1\}$, for some $T \unlhd G$, and so $G$ is abelian because $T \cong G/N \cong H$ is abelian, a contradiction.
\end{proof}

\begin{remark}
The converse of the first part of Theorem \eqref{nkb} is not true. For example, consider the group $G := \mathbb{Z}_{p} \rtimes \langle \sigma \rangle$, where $p > 2$ is a prime, $\sigma \in \mathrm{Aut}(\mathbb{Z}_{p})$ is of prime order, and $\langle \sigma \rangle$ acts on $\mathbb{Z}_{p}$ in an obvious way. Then we can conclude from the second part of Theorem \eqref{nkb} that $G$ contains a proper essential subgroup whereas both the components of $G$ do not have a proper essential subgroup.
\end{remark}

\section{\large{Essential extensions}}\label{mm}

\begin{definition}
A group $G$ is said to have an essential extension if there exists a monomorphism $\varphi: G \rightarrow \overline{G}$ such that $\varphi(G)$ is an essential subgroup of $\overline{G}$; if $\varphi(G) \neq \overline{G}$, the essential extension is said to be a proper essential extension.
\end{definition}

In the theory of essential submodules, the following equivalent conditions are well-known for an $R$-module $Q$: $Q$ does not have a proper essential extension if and only if, $Q$ is a direct summand of $M$ for every monomorphism $f: Q \rightarrow M$ of $R$-modules if and only if, $Q$ is injective. It is therefore very natural to ask whether we have an analogous equivalence for a group $G$ that does not have a proper essential extension. The following theorem answers this question affirmatively. 

\begin{theorem}\label{ma}
Let $G$ be a nontrivial group. Then the following statements are equivalent.

$(a)$ $G$ does not have a proper essential extension.

$(b)$ If $\varphi: G \rightarrow \overline{G}$ is a monomorphism with $\varphi(G) \unlhd \overline{G}$, then $\varphi(G)$ is a direct

\hspace{0.45cm} summand of $\overline{G}$.

$(c)$ $G$ is a complete group.

$(d)$ If $\varphi: G \rightarrow \overline{G}$ is a monomorphism with $\varphi(G) \unlhd \overline{G}$ and $\varphi(G) \neq \overline{G}$, then there

\hspace{0.45cm} exists a group $H$ and a morphism $g: \overline{G} \rightarrow H$ such that $g$ is not monic and

\hspace{0.45cm} $g \circ \varphi$ is monic.
\end{theorem}
\begin{proof}
$(a) \Rightarrow (b)$. Let $\varphi: G \rightarrow \overline{G}$ be a monomorphism such that $\varphi(G)$ is a normal subgroup of $\overline{G}$. We may suppose $\varphi(G) \neq \overline{G}$, then there exist nontrivial normal subgroups of $\overline{G}$ that intersect $\varphi(G)$ trivially; let $T$ be maximal respect to these conditions. Then we have that $(\varphi(G)T)/T$ is an essential subgroup of $\overline{G}/T$. In fact if $(\varphi(G)T)/T \cap K/T = \{1\}$ for some normal subgroup $K/T$ of $\overline{G}/T$, then $(\varphi(G)T) \cap K = T(\varphi(G) \cap K) \subseteq T$, hence $\varphi(G) \cap K = \{1\}$, and so $T = K$. Therefore $(\varphi(G)T)/T = \overline{G}/T$, because $(\varphi(G)T)/T \cong G$ does not have proper essential subgroups, and $\varphi(G)T = \overline{G}$.

While the proof of the assertion $(b) \Rightarrow (c)$ may be found in \cite{G}, we give an alternate independent proof here to make it more widely available. 

$(b) \Rightarrow (c)$ Embedding $G$ into its holomorph, we see that $\mathrm{Hol}(G) = GT$ and $G \cap T = \{1\}$, for some $T \unlhd \mathrm{Hol}(G)$.
Let $\sigma \in \mathrm{Aut}(G)$ and $\sigma = gt$, for some $g \in G$ and $t \in T$. For any $x \in G$, we have $\sigma(x) = \sigma x \sigma^{-1} = (gt)x(t^{-1}g^{-1}) = gxg^{-1}$, since $T$ centralizes $G$. Hence, every automorphism $\sigma$ of $G$ is inner.

We claim that if $\mathrm{Z}(G) \neq \{1\}$, then $\mathrm{Z}(G)$ is a divisible group. To prove the claim, assume on the contrary that $\mathrm{Z}(G)$ is not a divisible group. Therefore, there exists an abelian group $H$ containing $\mathrm{Z}(G)$ such that $H \neq \mathrm{Z}(G)N$, for all $N \leq H$ such that $\mathrm{Z}(G) \cap N = \{1\}$. If we form the central product $K$ of $G$ and $H$ by amalgamating $\mathrm{Z}(G)$, then we have $G \cap H = \mathrm{Z}(G)$ and $[G, H] = \{1\}$. Since $G \unlhd K$, the hypothesis implies that $K = GT$ and $G \cap T = \{1\}$, for some $T \unlhd K$. Since $H, T \leq \mathrm{C}_{K}(G)$ and since $K = GH = GT$, we get that $\mathrm{C}_{K}(G) = \mathrm{Z}(G)H= \mathrm{Z}(G)T$, and so $T \leq H$. This implies that $H = H \cap GT = \mathrm{Z}(G)T$ and $\mathrm{Z}(G) \cap T = \{1\}$, a contradiction.

Since Theorem \eqref{bchche} implies that either $\mathrm{Z}(G) = \{1\}$ or $\mathrm{Z}(G) \cong \mathbb{Z}_{2}$, and since $\mathrm{Z}(G)$ is a divisible group, therefore, we conclude that $\mathrm{Z}(G) = \{1\}$. 

$(c) \Rightarrow (a)$. Let $\varphi: G \rightarrow \overline{G}$ be a monomorphism such that $\varphi(G)$ is an essential subgroup of $\overline{G}$. In particular, $\varphi(G)$ is a normal subgroup of $\overline{G}$. Since $G$ is complete and $\varphi(G) \cong G$, so $\varphi(G)$ is complete. If $T$ is the centralizer of $\varphi(G)$ in $\overline{G}$, then it is well-known that  $T \unlhd \overline{G}$, $\overline{G} = \varphi(G)T$, and $\varphi(G) \cap T = \{1\}$. But this implies $T = \{1\}$, since $\varphi(G)$ is an essential subgroup of $\overline{G}$, and so $\overline{G} = \varphi(G)$.

$(a) \Rightarrow (d)$. Since $\varphi(G)$ is not an essential subgroup of $\overline{G}$, there exists a nontrivial proper normal subgroup $N$ of $\overline{G}$ such that $\varphi(G) \cap N = \{1\}$. Consider the canonical map $\pi: \overline{G} \rightarrow \overline{G}/N$. As $N \neq \{1\}$, $\pi$ is not a monomorphism. If $\pi(\varphi(x)) = N$, then $\varphi(x) \in N$, and this implies $x = 1$, since $\varphi(G) \cap N = \{1\}$ and $\varphi$ is a monomorphism. Hence, $\pi \circ \varphi$ is a monomorphism. 

$(d) \Rightarrow (a)$ Assume on the contrary that $\varphi: G \rightarrow \overline{G}$ is a proper essential extension of $G$. Then, for some group $H$, there exists a morphism $g: \overline{G} \rightarrow H$ such that $g \circ \varphi$ is a monomorphism and $g$ is not a monomorphism. Hence, $\mathrm{Ker}(g) \neq \{1\}$. As $\varphi(G)$ is an essential subgroup of $\overline{G}$, there exists $y \in \varphi(G) \cap \mathrm{Ker}(g)$ such that $y \neq 1$. Let $\varphi(x) = y$, for some $x \in G$. Since $g(\varphi(x)) = 1$ and $g \circ \varphi$ is a monomorphism, we get that $x = 1$. But this implies $y = 1$, a contradiction.
\end{proof}

Since divisible groups are precisely injective $\mathbb{Z}$-modules, they do not possess a proper essential extension in the category of abelian groups. If $G \neq \{1\}$ is an abelian group, then it is not complete and so applying Theorem \eqref{ma} we obtain Corollary \eqref{omnamosibai}. Thus, groups $\overline{G}$ that contain divisible groups as proper essential subgroups must all be nonabelian groups. 

\begin{corollary}\label{omnamosibai}
Every nontrivial abelian group $G$ has a proper essential extension.
\end{corollary}
\begin{proof}
Follows from Theorem \eqref{ma}.
\end{proof}

Let $G$ be a group and $\mathrm{Hol}(G)$ be its holomorph. We denote $\mathrm{Hol}^{n}(G)$ to mean the group $\mathrm{Hol}(\mathrm{Hol}( \cdots (\mathrm{Hol}(G)) \cdots ))$, where the parentheses occur $n$-times.

\begin{theorem}\label{mgrkhb}
For every group $G$, there exists a group $H$ such that $G$ is embedded in $H$ as a normal subgroup and $\mathrm{Hol}^{n}(H)$ contains a proper essential subgroup, for all $n \in \mathbb{N}$.
\end{theorem}
\begin{proof}
The statement is obvious if $G$ is the trivial group. Assume now that $|G| > 1$. If $G$ is a complete group, then $H = G \times G$ is not complete, since the automorphism $\sigma: H \rightarrow H$ given by $(a, b) \mapsto (b, a)$ is not inner, because if $\sigma$ would be inner, then for some $(\alpha, \beta) \in H$, we would have $(\alpha, \beta)(x, 1)(\alpha^{-1}, \beta^{-1}) = (1, x)$, for all $x \in G$, and this would imply $x = 1$, a contradiction, since $G$ is nontrivial. Also, $\mathrm{Z}(H) = \{1\}$, since $G$ is complete, and so $|\mathrm{Z}(H)| \neq 2$. For some $n \in \mathbb{N}$, if $\mathrm{Hol}^{n}(H)$ does not have a proper essential subgroup, then a repeated application of Theorem \eqref{kk} shows that $\mathrm{Hol}(H)$ does not have proper essential subgroups, and so we get from Theorem \eqref{kk} that $H$ is a direct summand of $\mathrm{Hol}(H)$. Now, Theorem \eqref{bchche} implies that either $H$ is complete or $|\mathrm{Z}(H)| = 2$, a contradiction.

Assume now that $G$ is not a complete group. Then Theorem \eqref{ma} implies that $G$ has a proper essential extension, say $H$. Now Theorem \eqref{nkb} implies that $\mathrm{Hol}(H)$ has a proper essential subgroup. Finally, Theorem \eqref{kk} implies that $\mathrm{Hol}^{n}(H)$ has a proper essential subgroup, for all $n \in \mathbb{N}$.
\end{proof}

\begin{remark}\label{rs}
In Theorem \eqref{mgrkhb}, we have used the fact that follows easily from Theorem \eqref{bchche} that for a nontrivial group $G$, if $|\mathrm{Z}(G)| \neq 2$ and if $G$ is not complete, then $\mathrm{Hol}(G)$ has a proper essential subgroup. These conditions can not be dropped, because, if we consider the Mathieu group $M_{11}$, then $\mathrm{Hol}(M_{11}) \cong M_{11} \times M_{11}$. Also, since $\mathrm{Hol}(\mathbb{Z}_{2}) \cong \mathbb{Z}_{2}$, we see that the converse is not true.
\end{remark}

\section{\large{A short proof of Eilenberg and Moore's theorem}}\label{mbmba}

An $R$-module $Q$ is injective if, for any monomorphism $f: L \rightarrow M$, and any morphism $\varphi: L \rightarrow Q$, there is a morphism $h: M \rightarrow Q$ such that $h \circ f = \varphi$. Motivated by this definition, we define extensional objects in the following way. 

\begin{definition}
A group $I$ in the category of groups is called an extensional object if, for any monomorphism $f: N \rightarrow G$ such that $f(N) \unlhd G$, and any morphism $\varphi: N \rightarrow I$, there is a morphism $h: G \rightarrow I$ such that $h \circ f = \varphi$. 
\end{definition}

It is natural to ask whether any extensional object exists in the category of groups. Theorem \eqref{ll} shows that such an extensional object in the category of groups must be the trivial group. On the other hand, recall that an object $\mathcal{I}$ in the category of groups is called injective if, for any monomorphism $K \rightarrow L$, and any morphism $K \rightarrow \mathcal{I}$, there is a morphism $L \rightarrow \mathcal{I}$ such that the following diagram commutes.

\[\begin{tikzcd}
& {\mathcal{I}} \\
K && L
\arrow[from=2-1, to=1-2]
\arrow[from=2-1, to=2-3]
\arrow[dashed, from=2-3, to=1-2]
\end{tikzcd}\]

\vspace{0.3cm}

The only injective object in the category of groups is the trivial group, a well-known result due to Eilenberg and Moore \cite{E}. Observe that an injective object in the category of groups is necessarily an extensional object in the category of groups. Therefore, proving that the only extensional object in the category of groups is the trivial group yields an alternative proof of Eilenberg and Moore's theorem. Another alternative proof of Eilenberg and Moore's theorem can be found here \cite{F}.

\begin{theorem}\label{ll}
The only extensional object in the category of groups is the trivial group.
\end{theorem}
\begin{proof}
On the contrary, assume that $I$ is an extensional object and $I \neq \{1\}$. From Theorem \eqref{ma} and Figure \eqref{pp}, it is easy to see that an extensional object is complete. From Figure \eqref{maa}, it is not hard to see that $H = I \times I$ is a nontrivial extensional object, so $H$ is complete. Since the automorphism $\sigma: H \rightarrow H$ given by $(a, b) \mapsto (b, a)$ is not inner, so $\mathrm{Out}(H) \neq \{1\}$, a contradiction. Thus, $I = \{1\}$.
\end{proof}

\begin{equation}\label{pp}
\begin{tikzcd}
& I \\
I && G
\arrow["{\text{identity}}", from=2-1, to=1-2]
\arrow["f", from=2-1, to=2-3]
\arrow["{\text{lift}}"', from=2-3, to=1-2]
\end{tikzcd}
\end{equation}

\begin{equation}\label{maa}
\begin{tikzcd}
H &&& I \\
&&&&&& H \\
&&&&& N && G \\
N &&& G
\arrow["{\pi_{1}}", shift left, from=1-1, to=1-4]
\arrow["{\pi_{2}}"', shift right, from=1-1, to=1-4]
\arrow["\varphi", from=3-6, to=2-7]
\arrow["f"', from=3-6, to=3-8]
\arrow["{(h_{1}, h_{2})}"', from=3-8, to=2-7]
\arrow["\varphi", from=4-1, to=1-1]
\arrow["{\pi_{2}\varphi}", shift left, from=4-1, to=1-4]
\arrow["{\pi_{1}\varphi}"', shift right, from=4-1, to=1-4]
\arrow["f"', from=4-1, to=4-4]
\arrow["{h_{1}}"', shift right, from=4-4, to=1-4]
\arrow["{h_{2}}", shift left, from=4-4, to=1-4]
\end{tikzcd}
\end{equation}

\begin{corollary}\label{rmrm}
The only injective object in the category of groups is the trivial group.
\end{corollary}
\begin{proof}
Follows from Theorem \eqref{ll}.
\end{proof}

\begin{remark}
Another short proof of Theorem \eqref{ll} can be the following: Let $\mathcal{F}[s, t]$ be the free group on two letters and $\sigma \in \mathrm{Aut}(\mathcal{F}[s, t])$ be such that $\sigma(s) = t$ and $\sigma(t) = s$. Consider the semi-direct product $\mathcal{F}[s, t] \rtimes \langle \sigma \rangle$, where $\langle \sigma \rangle$ acts on $\mathcal{F}[s, t]$ in an obvious way, and the monomorphism $f: \mathcal{F}[s, t] \rightarrow \mathcal{F}[s, t] \rtimes \langle \sigma \rangle$, defined by, $f(g) = g$ for all $g \in \mathcal{F}[s, t]$. Let $x \in I$ and $\varphi: \mathcal{F}[s, t] \rightarrow I$ is the morphism defined by $\varphi(s) = 1$ and $\varphi(t) = x$. Hence, a morphism $h: \mathcal{F}[s, t] \rtimes \langle \sigma \rangle \rightarrow I$ exists such that $h \circ f = \varphi$. As $\sigma t \sigma^{-1} = \sigma t \sigma = s$, therefore, $h(\sigma) (h \circ f)(t) h(\sigma) = (h \circ f)(s)$. This implies $h(\sigma) x h(\sigma) = 1$, and so $x = 1$. As $x \in I$ was arbitrary, $I$ is the trivial group.
\end{remark}

\section*{\large{Acknowledgement}}

The first author would like to express his deepest gratitude to the late Tapati Chowdhury and Neem Karoli Baba.

\bibliographystyle{amsplain}

\end{document}